\documentclass[leqno]{amsart}
\usepackage{amsmath}
\usepackage{amssymb}
\usepackage{amsthm}
\usepackage{enumerate}
\usepackage[mathscr]{eucal}
\theoremstyle{plain}
\usepackage{tikz}
\newtheorem{theorem}{Theorem}[section]
\newtheorem{lemma}[theorem]{Lemma}
\newtheorem{prop}[theorem]{Proposition}
\theoremstyle{definition}
\newtheorem{definition}[theorem]{Definition}
\newtheorem{remark}[theorem]{Remark}

\newtheorem{example}[theorem]{Example}

\newtheorem{cor}[theorem]{Corollary}
\theoremstyle{remark}
\begin{document}
	\title[On the best coapproximation problem  in $\ell_1^n$]{On the best coapproximation problem  in $ \ell_1^n$}
	\author[Sain, Sohel, Ghosh and  Paul  ]{Debmalya Sain, Shamim Sohel, Souvik Ghosh and Kallol Paul }

	\newcommand{\acr}{\newline\indent}
	\address[Sain]{Departamento de Analisis Matematico\\ Universidad de Granada\\ E-18071- Granada \\SPAIN}
	\email{saindebmalya@gmail.com}
	
	\address[Sohel]{Department of Mathematics\\ Jadavpur University\\ Kolkata 700032\\ West Bengal\\ INDIA}
	\email{shamimsohel11@gmail.com}
	
	\address[Ghosh]{Department of Mathematics\\ Jadavpur University\\ Kolkata 700032\\ West Bengal\\ INDIA}
	\email{sghosh0019@gmail.com}

	\address[Paul]{Department of Mathematics\\ Jadavpur University\\ Kolkata 700032\\ West Bengal\\ INDIA}
	\email{kalloldada@gmail.com}

	\thanks{The research of Dr. Debmalya Sain is supported by grant PID2021-122126NB-C31 funded  by MCIN/AEI/10.13039/501100011033 and by ``ERDF A way of making Europe" under the mentorship of Professor Miguel Martín.  The second and third author would like to thank  CSIR, Govt. of India, for the financial support in the form of Junior Research Fellowship under the mentorship of Prof. Kallol Paul.} 
	\subjclass{Primary 46B20, Secondary 47L05}
	\keywords{best coapproximations; coproximinal subspace; co-Chebyshev subspace;  Birkhoff-James orthogonality; dual space}

	\begin{abstract}
		We study the best coapproximation problem in the Banach space  $ \ell_1^n, $ by using Birkhoff-James orthogonality techniques. Given a subspace $\mathbb{{Y}}$ of $\ell_1^n$, we completely identify the elements $x$ in $\ell_1^n,$ for which best coapproximations to $x$ out of $\mathbb{{Y}}$ exist. The methods developed in this article are computationally effective and it allows us to present an algorithmic approach to the concerned problem. We also identify  the coproximinal subspaces and co-Chebyshev subspaces of $\ell_1^n$. 
	\end{abstract}
	
	\maketitle
	\section{Introduction}

	The concept of \emph{best coapproximation} in Banach space  has been introduced by 	Franchetti and Furi in \cite{FF}. Thereafter, this complementary notion of the best approximation has been studied by many mathematicians \cite{FF, N, PS, RS}.	 Finding the best coapproximation  in Banach spaces is known to be a difficult problem especially from the computational point of view. Very recently some progress has been made in \cite{SSGP}, where the problem was completely characterized in \emph{the space of diagonal matrices} and in the same article  the authors solved the best coapproximation problem computationally in the subspaces of $ \ell_\infty^n.$  An interesting natural choice for further study in this direction is  to consider the best coapproximation problem in the subspaces of the dual space $\ell_1^n.$  It is known \cite {LT, PS} that  given a subspace $\mathbb{{Y}} $ of a Banach space $ \mathbb{X}$ and an element $ x \notin \mathbb{{Y}},$  $y_0$ is a best coapproximation to $x$ out of $\mathbb{{Y}}$ if and only if there exists a norm one projection from $ span \, \{x, \mathbb{{Y}}\} $ to $\mathbb{{Y}}.$ On the other hand, a little checking on part of the reader should suffice to convince that the above mentioned theoretical characterization is not particularly effective in explicitly finding the best coapproximation(s), if it exist.  Our main aim in this article is to deal with  the best coapproximation problem in $\ell_1^n$ from a computational point of view and to solve the problem explicitly. We completely identify the subspaces of $\ell_1^n$  which are coproximinal and co-Chebyshev.  It is time to mention the basic terminologies and the notations to be used throughout the article.
	
	\smallskip

	We use the symbols $ \mathbb{X}, \mathbb{{Y}}$ to denote real Banach spaces, unless stated otherwise. Let $\theta$ denote the zero vector of any Banach space, other than the scalar field. The usual notations  $ B_{\mathbb{X}} = \{ x \in \mathbb{X} : \| x\| \leq 1\}$ and $ S_{\mathbb{X}} = \{ x \in \mathbb{X} : \| x\| = 1\} $ are used to denote the unit ball and the unit sphere of $ \mathbb{X}, $ respectively. An element $ x \in B_{\mathbb{X}}$ is said to be an extreme point of the unit ball if $ x = (1-t)y +tz, $ for some $ t \in (0,1)$ and for some $y,z \in B_{\mathbb{X}},$  then $ x = y =z.$ The collection of all extreme points of the unit ball $ B_{\mathbb{X}} $  is denoted by $Ext(B_{\mathbb{X}}).$ The  dual of a Banach space $ \mathbb{X}$ is denoted by  $ \mathbb{X}^*.$  Given any $f \in \mathbb{X}^*,$  $ M_f $ denotes the norm attainment set of $f$, i.e.,$ M_f:=\{ x \in S_{\mathbb{X}} : |f(x)|= \|f\|\}.$ We note that $M_f \neq \emptyset,$ whenever $\mathbb{X}$ is reflexive. 
	Given any $ m \times n$ matrix $A$,  $ A^t$ denotes the transpose of $A.$ Let $\mathbb{L}(\mathbb{X}, \mathbb{Y})$  be denoted as  the Banach space of all bounded linear operators from $\mathbb{X}$ to $\mathbb{Y},$ endowed with the usual operator norm. For given any $ T \in \mathbb{L}(\mathbb{X},\mathbb{Y}),$ the kernel of $T,$ denoted by $ker~T,$ is defined as $ker~T:=\{ x \in \mathbb{X} : Tx= \theta \in \mathbb{Y}\}.$ Accordingly, the kernel of $f \in \mathbb{X}^*$ is denoted by $ker~f,$ i.e., $ker~f = \{x \in \mathbb{X} : f(x) = 0\}.$   Let us now recall the following definition of best coapproximation which is of fundamental importance in our context.
	
	\begin{definition}
		Let $ \mathbb{X} $ be a Banach space and let $ \mathbb{Y} $ be a subspace of $ \mathbb{X}. $ Given any $ x \in \mathbb{X}, $ we say that $ y_0 \in \mathbb{Y} $ is a best coapproximation to $ x $ out of $ \mathbb{{Y}} $ if $ \| y_0 - y \| \leq \| x - y \| $ for all $ y \in \mathbb{Y}. $
	\end{definition}
	
	In a Banach space (even in the finite-dimensional case), neither the  existence nor the uniqueness of best coapproximation(s) is guaranteed. A subspace $ \mathbb{Y} $ of the Banach space $ \mathbb{X} $ is said to be coproximinal if a best coapproximation to any element of $ \mathbb{X} $ out of $ \mathbb{Y} $ exists. A coproximinal subspace $ \mathbb{Y} $ is said to be co-Chebyshev if the best coapproximation is unique in each case. Given $x \in \mathbb{X}$ and a subspace $\mathbb{{Y}}$ of $ \mathbb{X}, $ we denote by $\mathcal{R}_\mathbb{{Y}}(x)$ the set of all best coapproximations to $x$ out of $ \mathbb{{Y}}.$ We also define $\mathcal{D}(\mathcal{R}_{\mathbb{Y}})$ as the collection of all such $x \in \mathbb{X}$ such that $\mathcal{R}_{\mathbb{Y}}(x) \neq \emptyset.$ 
	
	\smallskip
	
	The study of best coapproximation has an immediate connection to \emph{the concept of Birkhoff-James orthogonality.} Following the pioneering articles \cite{B, J}, given any two elements $ x, y $ in a Banach space $ \mathbb{X}, $ we say that \emph{$ x $ is Birkhoff-James orthogonal to $ y, $} written as $ x \perp_B y, $ if $ \| x+\lambda y \| \geq \| x \| $ for all $ \lambda \in \mathbb{R}. $ The said connection can be stated (and verified in a rather straightforward manner) in terms of Birkhoff-James orthogonality as follows:  Given a subspace $ \mathbb{Y} $ of a Banach space $ \mathbb{X} $ and an element $ x \in \mathbb{X}, $ $ y_0 \in \mathbb{Y} $ is a best coapproximation to $ x $ out of $ \mathbb{Y} $ if and only if $ \mathbb{Y} \perp_B (x-y_0) $ i.e., $ y \perp_B (x-y_0) $ for all $ y \in \mathbb{Y}. $ In this article, our main objective is to completely solve the problem of finding the best coapproximation(s) to a given element in $  \ell_1^n$ out of a given subspace $ \mathbb{Y} $ of $  \ell_1^n$, provided the best coapproximation(s) exist. As mentioned before, it is known that given a subspace $\mathbb{{Y}} $ of $ \ell_1^n$ and $ x \notin \mathbb{{Y}},$  $y_0$ is a best coapproximation to $x$ out of $\mathbb{{Y}}$ if and only if there exists a norm one projection from $ span \, \{x, \mathbb{{Y}}\} $ to $\mathbb{{Y}}.$  However, to the best of our knowledge, there is no method available to explicitly find these norm one projections.  We present a computationally effective solution to this problem, resulting in an algorithmic approach to the best coapproximation problem in $ \ell_1^n.$ It also allows us to discuss the existence of best coapproximations in the said setting. We also  completely identify  the coproximinal and co-Chebyshev subspaces of $  \ell_1^n.$

	\section{Main Results}

		 We first  introduce the following definition.

		\begin{definition} \label{component}
		
		Let $ \mathcal{A} = \{  \widetilde{a_1}, \widetilde{a_2}, \ldots, \widetilde{a_m} \}  $ be a set of linearly independent elements in $ \ell_1^n, $  where  $\widetilde{a_k}= (a^k_1, a^k_2, \ldots, a^k_n ),  $ for each $ 1 \leq k \leq m. $  Considering   $  \widetilde{a_1}, \widetilde{a_2}, \ldots, \widetilde{a_m} $ as column vectors,   we form the $ n \times m $ matrix $ A= (a_{ij})_{ 1 \leq i \leq n, 1 \leq j \leq m},$ where $ a_{ij}= a_i^j.$  
		\begin{enumerate}
			\item [(i)]  For each $ i \in \{1,2,\ldots,n\}, $ the $i$-th component  of  $\mathcal{A}$ is defined  as the $i$-th row of $A,$ i.e., 
			$\left( a_{i} ^1 ,a_{i}^2 ,\ldots,a_{i} ^m \right). $ Whenever the context is clear, we simply refer to the $i$-th component of $\mathcal{A}$ as the $i$-th component.
			\item  [(ii)]  The $i$-th component and the $j$-th component are said to  be \emph{equivalent} if 
				$\left( a_{i} ^1 ,a_{i}^2 ,\ldots,a_{i} ^m \right) $ = $c \left( a_{j} ^1 ,a_{j}^2 ,\ldots,a_{j} ^m \right),  $ where $~c( \neq 0 ) \in \mathbb{R}.$
	
			\item[(iii)] The \emph{zero set} $ \mathcal{Z}_{\mathcal{A}}$ of $\mathcal{A}$ is defined as 
			\[
			\mathcal{Z}_{\mathcal{A}} = \Big\{  i \in \{1, 2, \ldots, n\} : \left( a_{i} ^1 ,a_{i}^2 ,\ldots,a_{i} ^m \right) =  (0, 0, \ldots, 0) \Big\}.
			\]
		\end{enumerate}

	\end{definition}

 In the following proposition, we show the basis invariance of \emph{the equivalent components and the zero set.}

		\begin{prop}\label{propo}
		Let $ \mathbb{Y} $ be a subspace of $ \ell_1^n$ and let $\mathcal{A} = \{  \widetilde{a_1}, \widetilde{a_2}, \ldots, \widetilde{a_m} \}, \mathcal{B}= \{  \widetilde{b_1}, \widetilde{b_2}, \ldots, \widetilde{b_m} \}$ be two bases of $ \mathbb{Y},$
		where $ \widetilde{a_k}= (a^k_{1}, a^k_{2}, \ldots, a^k_{n})  $ and $ \widetilde{b_k} = (b^k_{1}, b^k_{2}, \ldots, b^k_{n}) $, for any $ 1 \leq k \leq m .$  Then 
		\begin{itemize}
		
		\item[(i)] the $i$-th and $j$-th component of $\mathcal{A}$ are equivalent if and only if the $i$-th and $j$-th component of $\mathcal{B}$ are equivalent.
			\item[(ii)] $\mathcal{Z}_{\mathcal{A}} = \mathcal{Z}_{\mathcal{B}}.$
		\end{itemize}
	\end{prop}

	\begin{proof}
		(i)	 Consider the two matrices  $A$ and $B$ as constructed  in Definition \ref{component}. Since $\mathcal{A}$ and $ \mathcal{B} $ are two bases of $\mathbb{Y},$  there exists an invertible matrix $ C= (c_{ij})_{1 \leq i, j \leq m}$ such that $ B= AC,$ where $ b_{ij}= \sum_{k=1}^{m} a_{ik} c_{kj},$ for any $ 1 \leq i \leq n, 1 \leq j \leq m.$
		The $j$-th components of $\mathcal{B} $, 
		\begin{eqnarray}
			(b_j^1, b_j^2, \ldots, b_j^m) = \bigg( \sum_{k=1}^mc_{k1}a_j^k,  \sum_{k=1}^mc_{k2}a_j^k, \ldots,  \sum_{k=1}^mc_{km}a_j^k \bigg).
		\end{eqnarray}
		Suppose that the $i$-th and the $j$-th components of $\mathcal{A}$ are equivalent.  Then  $(a_j^1, a_j^2, \ldots, \\a_j^m) = c(a_i^1, a_i^2, \ldots, a_i^m), $ for some $c \in \mathbb{R}.$ Therefore, 
		\[
		(b_j^1, b_j^2, \ldots, b_j^m) = c \bigg( \sum_{k=1}^mc_{k1}a_i^k,  \sum_{k=1}^mc_{k2}a_i^k, \ldots,  \sum_{k=1}^mc_{km}a_i^k \bigg) = c(b_i^1, b_i^2, \ldots, b_i^m ).
		\] 
		This implies that  the $i$-th and the $j$-th components of $\mathcal{B}$ are equivalent. By	a similar argument,  we can easily obtain the converse result. \\
		
		(ii) Follows immediately from  equation (1). 
	\end{proof}

	We next obtain a  characterization of best coapproximation(s) in finite-dimensional subspaces of the dual of a reflexive Banach space. This simple observation will play an important role in finding the best coapproximation in the subspaces of $\ell_1^n.$ 
	
	\begin{theorem}\label{general}
		Let $\mathbb{X}$ be a reflexive Banach space and let $ g_1, g_2, \ldots, g_m \in \mathbb{X}^*$ be linearly independent. Given any $\alpha_1, \alpha_2, \ldots, \alpha_m \in \mathbb{R},$ $\sum_{k=1}^m \alpha_kg_k$ is a best coapproximation to $ f \in \mathbb{X}^*$ out of $span\{g_1, g_2, \ldots, g_m\}$ if and only if given any $\beta_1, \beta_2, \ldots, \beta_m \in \mathbb{R},$
		\[ M_{\sum_{k=1}^m\beta_kg_k} \cap ker (f-\sum_{k=1}^m\alpha_kg_k) \neq \emptyset \]
		
	\end{theorem}
	\begin{proof}
		It follows from the definitions of Birkhoff-James orthogonality and best coapproximation that $ \sum_{k=1}^{m} \alpha_kg_k $ is a best coapproximation to $ f $ out of $ span \{ g_1, g_2, \ldots, g_m \} $ if and only if $ g \perp_B \left(f - \sum_{k=1}^{m} \alpha_kg_k\right), $ for all $ g \in span \left\lbrace  g_1, g_2, \ldots, g_m \right\rbrace. $ Clearly, this is equivalent to the following:
		
		\[ \sum_{k=1}^{m} \beta_kg_k \perp_B (f-\sum_{k=1}^{m} \alpha_kg_k) ~\forall ~\beta_1,\beta_2, \ldots,\beta_m \in \mathbb{R}. \]
		Now applying \cite[Th. $3.2$]{S}, we conclude that the above condition is equivalent to $ M_{\sum_{k=1}^m\beta_kg_k} \cap ker (f-\sum_{k=1}^m\alpha_kg_k) \neq \emptyset$. This completes the proof of the theorem.
		
	\end{proof}

	In the following section, we focus on subspaces $\mathbb{Y} $ of $\ell_1^n$ spanned by the basis $\mathcal{A}$ with $\mathcal{Z}_{\mathcal{A}} = \emptyset.$

	\section*{Section-I}

	We begin this section by noting that there exists a canonical isometric isomorphism $\psi$ between $ \ell_1^n $ and $ (\ell_\infty^n)^*,$  defined as $\psi(a_1,a_2, \ldots,a_n) =g,$ where $ g : \ell_\infty^n \to \mathbb{R}$ is given as : $ g( \alpha_1 e_1 + \alpha_2e_2 + \ldots \alpha_ne_n) = \alpha_1a_1 + \alpha_2 a_2 + \ldots \alpha_na_n, $  $ \{e_1, e_2, \ldots, e_n\}$ being the standard ordered basis of  $\mathbb{R}^n.$  Thus, given a subspace $\mathbb{{Y}} $ of $ \ell_1^n$ and an element $x \notin \mathbb{{Y}},$ the problem of finding best coapproximation to $x$ out of $\mathbb{{Y}}$ is equivalent to the problem of finding the same to $\psi(x)$ out of the subspace $\psi(\mathbb{{Y}})$ in $ (\ell_\infty^n)^*.$ This observation will be used as and when required.	We also require  the following two definitions.

%	Let  $\mathbb{{V}}$  be a subspace of $(\ell_\infty^n)^*$ such that $\mathcal{Z}_\mathcal{A}= \phi,$  where $\mathcal{A}$ is a basis of $\mathbb{{V}}.$  We characterize the best coapproximation to a given element in $(\ell_\infty^n)^*$ out of a subspace $\mathbb{V}$ of $(\ell_{\infty}^n)^*.$

	\begin{definition}
		A set $S$ in a Banach space is said to be   symmetric  if $x \in S$ implies  $-x \in S.$
	\end{definition}
	
	\begin{definition}
		Let $\mathbb{{Y}}  $ be a subspace of $ \ell_1^n.$  A symmetric set $\mathcal{N}$ is said to be 
		a norming set  of $ \mathbb{{Y}}$ if  $\left (M_g \cap Ext(B_{\ell_\infty^n}) \right)  \cap  \mathcal{N} \neq \emptyset,$ for each  $g\in \psi(\mathbb{{Y}}).$ A norming set  $\mathcal{N}$  is said to be  a minimal norming set of  $ \mathbb{{V}}$ if for some norming set $\mathcal{M}$ of $ \mathbb{{Y}},$ $\mathcal{M} \subset \mathcal{N}$ implies that  $\mathcal{M}= \mathcal{N}.$
	\end{definition}
	
	Observe that $M_g \cap Ext(B_{\ell_\infty^n}) \neq \emptyset,$ for each  $g\in \psi(\mathbb{{Y}}).$  Clearly, the minimal norming set may not be unique.	Let  $ g  \in (\ell_\infty^n)^*.$ Then for any $x= (x_1, x_2, \ldots, x_n) \in \ell_\infty^n,$ $ g(x)= \sum_{i=1}^{n} g(e_i)x_i.$ 	The following result ensures the existence of the minimal norming set of a subspace $\mathbb{{Y}} $ of $ \ell_1^n. $\\
	
	\begin{theorem}\label{norming:set}
	Let $\mathbb{{Y}}  $ be a subspace of $ \ell_1^n $  and let $ \mathcal{A} = \{\widetilde{a_1}, \widetilde{a_2}, \ldots, \widetilde{a_m}\}  $  be a basis of $\mathbb{{Y}}  $ with $\mathcal{Z}_{\mathcal{A}} = \emptyset.$  Then there exists a unique minimal norming set of $\mathbb{Y}.$ 
	\end{theorem}

	\begin{proof}
		Suppose that $\psi(\widetilde{a_i})=g_i,$ for any $1 \leq i \leq m.$ Since $\mathcal{Z}_{\mathcal{A}}= \emptyset,$ we observe that 
			\[\bigg\{e_1, e_2, \ldots, e_n\bigg\} \bigcap \bigg(\bigcap_{j=1}^{m} Ker g_j\bigg)= \emptyset.\]
			 Any element $g$ of $\psi(\mathbb{Y})$ is of the form $ g = \sum_{k=1}^m \beta_k g_k,  $ where $( \beta_1, \beta_2, \ldots, \beta_m) \in \mathbb{R}^m.$ Moreover, we note that $g_k(e_i)=a_i^k,$ for any $ 1 \leq k \leq m, 1 \leq i \leq n.$   We will prove the theorem  in the following four  steps.\\
		
		\textbf{Step  1 : }  We express $\mathbb{R}^m$ as the union of finitely many hyperplanes and open sets which are relevant to our purpose.  For each $i = 1, 2, \ldots, n, $ we consider the hyperplane $H_i$ of $\mathbb{R}^m,$ given by
		\[ H_i = \bigg\{(\beta_1, \beta_2, \ldots, \beta_m)\in \mathbb{R}^m : \sum_{k=1}^m\beta_k g_k(e_i) = 0\bigg\}. \]
		Assume that $H_1,H_2, \ldots H_r$ are distinct hyperplanes, where $ r \leq n.$ 
		For each $i =1,2 \ldots r,$ consider the sets  $H_i^+ $ and $H_i^-$ given by 
		\[ H_i^+ = \bigg\{(\beta_1, \beta_2, \ldots, \beta_m)\in \mathbb{R}^m : \sum_{k=1}^m\beta_k g_k(e_i) > 0\bigg\}\]
		\[H_i^- = \bigg\{(\beta_1, \beta_2, \ldots, \beta_m)\in \mathbb{R}^m : \sum_{k=1}^m\beta_k g_k(e_i) < 0\bigg\}.
		\]
		Observe that   $H_i^+ \cap H_i^- = \emptyset$ and $H_i^+ \cup H_i \cup H_i^- = \mathbb{R}^m,$ for each $i=1,2,\ldots r.$  Consider the set  $K_j= H_1^{\delta_{j_1}} \cap H_2^{\delta_{j_2}} \cap \ldots \cap H_r^{\delta_{j_r}},$ where $\delta_{j_i} \in \{+,-\}$ for each $i=1,2, \ldots r.$  It is immediate that there are atmost $2^r$ number of  such  sets. We assume that   $\pm K_1, \pm K_2, \ldots, \pm K_q$ are the nonempty such sets. Then  
		\[\mathbb{R}^m =  \bigg(\cup_{i=1}^q ( K_{i}  \cup -K_{i})\bigg) \bigcup \bigg(\cup_{j=1}^r H_j\bigg) = K \cup H ,\] 
		where $ K = \cup_{i=1}^q ( K_{i}  \cup - K_{i}) $ and $ H = \cup_{j=1}^r H_j.$ \\

		\textbf{Step 2 :} We find the norm attaining set of functionals of the form $ \sum_{k=1}^m \beta_k g_k ,$ where  $ ( \beta_1, \beta_2, \ldots, \beta_n) \in  K.$ 
		We first  associate each of the nonempty  sets  $K_s (1 \leq s \leq q)$  with an extreme point of $ B_{\ell_{\infty}^n}.$ 
		For any $\widetilde{\beta}= (\beta_1, \beta_2, \ldots, \beta_m) \in K_s,$ 
		let us construct $ \widetilde{x_s} = (x_1, x_2, \ldots, x_n) \in S_{\ell_\infty^n} $ where 
		\begin{eqnarray*}
			x_t &= 1, \, & \widetilde{\beta} \in H_t^+  \\ 
			& = -1, \, & \widetilde{\beta} \in H_t^- .
		\end{eqnarray*}
		Clearly, $ \widetilde{x_s} \in Ext(B_{\ell_\infty^n}).$   Note that construction of $\widetilde{x_s}$ is independent of $\widetilde{\beta}$, for if  $\widetilde{\beta} \in (H_{i_1}^+\cap  H_{i_2}^+ \ldots \cap H_{i_s}^+) \cap (H_{j_1}^- \cap H_{j_2}^- \ldots \cap H_{j_t}^-), $ where $ 1 \leq s+t \leq r,$  then  for any $\widetilde{\omega} \in K_s,$ we have $\widetilde{\omega} \in (H_{i_1}^+\cap  H_{i_2}^+ \ldots \cap H_{i_s}^+) \cap (H_{j_1}^- \cap H_{j_2}^- \ldots \cap H_{j_t}^-). $ 
		Thus  with each $K_i$  we associate an extreme point  $\widetilde{x_i}, $ where $1 \leq i \leq q.$
		Let $\mathcal{N}=\{\pm \widetilde{x_1},\pm  \widetilde{x_2}, \ldots, \pm \widetilde{x_q} \}. $   We show that  $ \mathcal{N} = \cup_{ (\beta_1, \beta_2, \ldots, \beta_m) \in K} M_{\sum_{k=1}^m \beta_kg_k}.$  Let $ ( \beta_1, \beta_2, \ldots, \beta_m)  \in K,$  then $ ( \beta_1, \beta_2, \ldots, \beta_m)  \in K_s,$ for some $s.$  Consider $ g = \sum_{k=1}^{m} \beta_kg_k. $ Then 
		$g(\widetilde{x_s}) =   \sum_{k=1}^{m} \beta_k\\g_k(e_i)x_i > 0.$  We show that $ M_g = \{ \pm \widetilde{x_s} \}.$ 
		Let $\widetilde{y}=(y_1, y_2, \ldots, y_n) \in M_{g} \cap Ext(B_{\ell_\infty^n}).$  \\
		Therefore, 
		$  |g(\widetilde{y})|= 
		\bigg|\sum_{k=1}^m\beta_k g_k (\widetilde{y}) \bigg|= \bigg|\bigg(\sum_{k=1}^m\beta_k g_k(e_1)\bigg) \bigg|+ \bigg| \bigg(\sum_{k=1}^m\beta_k g_k(e_2)\bigg) \bigg| +\ldots + \bigg| \bigg(\sum_{k=1}^m\beta_k g_k(e_n) \bigg) \bigg|
		$ 
		which  implies 
		\begin{eqnarray*}\label{eqny}
			\bigg| \bigg(\sum_{k=1}^m\beta_k g_k(e_1) \bigg)y_1 + \bigg( \sum_{k=1}^m\beta_k g_k(e_2) \bigg)y_2 +\ldots + \bigg( \sum_{k=1}^m\beta_k g_k(e_n) \bigg)y_n \bigg|\\ = 
			\bigg|\bigg(\sum_{k=1}^m\beta_k g_k(e_1)\bigg) \bigg|+ \bigg| \bigg(\sum_{k=1}^m\beta_k g_k(e_2)\bigg) \bigg| +\ldots + \bigg| \bigg(\sum_{k=1}^m\beta_k g_k(e_n) \bigg) \bigg|.
		\end{eqnarray*}
		The last equality is satisfied  if and only if each $\bigg(\sum_{k=1}^m\beta_k g_k(e_i) \bigg)y_i $ have the same sign, which in turn is satisfied  if and only if $\widetilde{y}= \pm \widetilde{x_s}.$ Thus $ M_g = \{ \pm \widetilde{x_s} \}.$ 
		Therefore, whenever $(\alpha_1, \alpha_2, \ldots, \alpha_m) \in K_i,$ for some $1 \leq i \leq q,$ we have $M_{\sum_{k=1}^{m}  \alpha_k g_k} = \{ \pm \widetilde{x_i}\}.$  Thus $ \mathcal{N} = \cup_{ (\beta_1, \beta_2, \ldots, \beta_m) \in K} M_{\sum_{k=1}^m \beta_kg_k}.$ \\ 
		
		\textbf{Step 3 : } We deal with functionals of the form $ \sum_{k=1}^m \gamma_k g_k ,$ where $ \widetilde{\gamma} = ( \gamma_1, \gamma_2, \ldots, \\ \gamma_m) \in  H. $ Let us assume that 
		\[\widetilde{\gamma} \in H_{b_1} \cap H_{b_2} \cap \ldots \cap  H_{b_u} \cap H_{c_1}^+  \cap  H_{c_2}^+ \cap \ldots \cap H_{c_s}^+ \cap H_{d_1}^- \cap H_{d_2}^- \ldots \cap H_{d_t}^-,\]
		where  $ u +s +t =r.$ 
		
		Let us now consider the set $ D_{\widetilde{\gamma}} = (\cap_{i=1}^{s} H_{c_i}^+ ) \cap (\cap_{i=1}^{t} H_{d_i}^-).$  Now it is easy to observe that $\widetilde{\gamma} \in D_{\widetilde{\gamma}}$ and $D_{\widetilde{\gamma}}$ is an open set of $\mathbb{R}^m.$ 
		Take 
		\[\widetilde{\eta}=(\eta_1, \eta_2, \ldots, \eta_n) \in D_{\widetilde{\gamma}} \setminus (\cup_{i=1}^u H_{b_i}).\]
		Then $\sum_{k=1}^{m} \eta_k g_k(e_l) > 0,$ for any $l \in  \{ c_1, c_2, \ldots, c_s\}$ and $\sum_{k=1}^{m} \eta_k g_k(e_l) < 0,$ for any $l \in  \{ d_1, d_2, \ldots, d_t\}. $ It is easy to observe that $ \widetilde{\eta} \in K_p,$ for some $1 \leq p \leq q.$
		Therefore, $M_{\sum_{k=1}^{m} \eta_k g_k}= \{\pm \widetilde{x_p}= (x_1, x_2, \ldots, x_n)\}$  as claimed before.
		Observe that   $x_l=1,$ for any  $l \in  \{ c_1, c_2, \ldots, c_s\} $ and $x_l =-1,$ for all $ l \in \{ d_1, d_2, \ldots, d_t\}.$ 
		By a straightforward calculation it is easy to observe that $\pm \widetilde{x_p}\in M_{\sum_{k=1}^{m} \gamma_k g_k}.$\\
		
		\textbf{Step 4 : } We show that $\mathcal{N} $ is the unique minimal norming set of $\mathbb{Y}.$ From the previous two steps it follows that $\mathcal{N}$ is a norming set of $\mathbb{Y}.$ 	Let us consider a symmetric set $\mathcal{M} \subsetneq \mathcal{N}$ and also assume that $ \pm \widetilde{x_j} \in \mathcal{N} \setminus \mathcal{M}.$ It can be clearly seen that whenever $\widetilde{\beta} \in K_j,$ $\mathcal{M} \cap M_{\sum_{k=1}^{m}\beta_kg_k} = \emptyset.$ This implies that $\mathcal{N}$ is a minimal norming set of $\mathbb{Y}.$ From step 2 it follows that $\mathcal{N}$ is the unique minimal norming set of $\mathbb{Y}.$ 
		
	\end{proof}
	
	We next explore the converse of the previous result. 
	
	\begin{theorem}\label{unique}
		Let $\mathbb{{Y}}  $ be a subspace of $ \ell_1^n$  and let $ \mathcal{A} = \{  \widetilde{a_1}, \widetilde{a_2}, \ldots, \widetilde{a_m} \}  $  be a basis of $\mathbb{{Y}}.$  If the minimal norming set of $\mathbb{Y}$ is unique then  $\mathcal{Z}_{\mathcal{A}} = \emptyset.$
	\end{theorem}
	
	\begin{proof}
		Let us assume that the minimal norming set $\mathcal{N}$ of $\mathbb{Y}$ is unique. 
		Suppose on the contrary that $j \in \mathcal{Z}_{\mathcal{A}},$ for some $ 1 \leq j \leq n.$ This  implies that $g(e_j)=0,$ for any $g \in \psi(\mathbb{Y}).$ Suppose that $\mathcal{N}= \{ \pm \widetilde{x_1}, \pm \widetilde{x_2}, \ldots, \pm \widetilde{x_q} \}$ is a minimal norming set of $\mathbb{Y},$ where $\widetilde{x_k}= ( x_1^k, x_2^k, \ldots, x_n^k),$ for $1 \leq k \leq q.$ Let us now consider 
		\[\widetilde{y_1} = (x_1^1, x_2^1,  \ldots, x_{j-1}^1, -x_j^1, x_{j+1}^1, \ldots, x_n^1)\in Ext(B_{\ell_\infty^n}).\]
		It can be easily observed that for any $g \in \psi(\mathbb{Y}),$ $\widetilde{y_1} \in M_g$ if and only if $\widetilde{x_1} \in M_g.$ 
		Therefore, $\mathcal{N}_1= \{\pm \widetilde{y_1},  \pm \widetilde{x_2}, \ldots, \pm \widetilde{x_q} \}$ is a norming set of $\mathbb{Y}.$ Since $\widetilde{y_1} \notin \{\pm \widetilde{x_2}, \ldots, \pm \widetilde{x_q} \},$ $\mathcal{N}_1 (\neq \mathcal{N})$ is a minimal norming set of $\mathbb{Y}.$ This contradicts the assumption that the minimal norming set of $ \mathbb{Y}$ is unique.
	\end{proof}

	Now we are in a position to present the characterization of the best coapproximation in  $ \mathbb{{Y}}.$ This is given in terms of a system of linear equations which clearly illustrates its computational effectiveness.
	
	\begin{theorem}\label{characterization}
	Let $\mathbb{{Y}}  $ be a subspace of $ \ell_1^n $  and let $ \mathcal{A} = \{  \widetilde{a_1}, \widetilde{a_2}, \ldots, \widetilde{a_m} \}  $  be a basis of $\mathbb{{Y}}  $ with $\mathcal{Z}_{\mathcal{A}} = \emptyset.$ 
		Suppose that $ \mathcal{N} = \{ \pm \widetilde{x_1}, \pm \widetilde{x_2},\ldots, \pm \widetilde{x_q} \}$ is  the minimal norming set of $\mathbb{{Y}},$ where $\widetilde{x_k} = (x_1^k, x_2^k,  \ldots, x_n^k),$ for any $1\leq k\leq q.$ Then given  $ \widetilde{b}=(b_1, b_2, \ldots, b_n) \in \ell_1^n,$ $\sum_{k=1}^m\alpha_k \widetilde{a_k} $ is a best coapproximation to $\widetilde{b}$ out of $\mathbb{{Y}}$ if and only if $\alpha_1, \alpha_2, \ldots, \alpha_m \in \mathbb{R}$ satisfy the following relations:\\
		\begin{eqnarray*}
			\alpha_1\sum_{i=1}^n a_i^1 x_i^p + \alpha_2\sum_{i=1}^n a_i^2 x_i^p + \ldots + \alpha_m\sum_{i=1}^n a_i^m x_i^p = \sum_{i=1}^n b_i x_i^p,
		\end{eqnarray*}
		for any $p \in \{ 1, 2, \ldots, q\}.$ 
	\end{theorem}
	
	\begin{proof}
		Suppose that $\psi(\widetilde{a_i})=g_i,$ for any $ 1 \leq i \leq m$ and  $ \psi(b)=f.$  We observe that $\sum_{k=1}^{m} \alpha_k \widetilde{a_k}$ being a best coapproximation to $\widetilde{b}$ out of $\mathbb{{Y}}$ is equivalent to $\sum_{k=1}^m\alpha_k g_k $ being a best coapproximation to $f$ out of $\psi(\mathbb{{Y}}).$
		
		\smallskip
		
		Let us first prove the necessary part of the theorem. Since
		$\{ \pm \widetilde{x_1}, \pm \widetilde{x_2},\ldots, \pm \widetilde{x_q} \} $ is the minimal norming set of $\mathbb{Y},$ it can be easily observed that  for any $\widetilde{x_s},$ there exists $\widetilde{\beta}=(\beta_1, \beta_2, \ldots, \beta_m)\in \mathbb{R}^m$ such that 
		$ M_{\sum_{k=1}^{m} \beta_k g_k} = \{\pm \widetilde{x_s}\}.$
 It follows from Theorem \ref{general} that $M_{\sum_{k=1}^m\beta_kg_k} \cap ker (f - \sum_{k=1}^m\alpha_kg_k) \neq \emptyset.$ 
		Therefore, $\widetilde{x_s} \in ker(f - \sum_{k=1}^m\alpha_kg_k),$ i.e., $(f - \sum_{k=1}^m\alpha_kg_k)\widetilde{x_s} = 0, $ which implies
		\[ f(\widetilde{x_s}) = \alpha_1g_1(\widetilde{x_s}) + \alpha_2g_2(\widetilde{x_s}) + \ldots + \alpha_mg_m(\widetilde{x_s}).\]
		This is equivalent to
		\[
		\alpha_1\sum_{i=1}^n g_1(e_i) x_i^s + \alpha_2\sum_{i=1}^n g_2(e_i)x_i^s + \ldots + \alpha_m\sum_{i=1}^n g_m(e_i)x_i^s = \sum_{i=1}^n f(e_i)x_i^s.\]
		Similarly, we can observe that for all $p \in \{1, 2, \ldots, q\},$
		\[ 	\alpha_1\sum_{i=1}^n g_1(e_i) x_i^p + \alpha_2\sum_{i=1}^n g_2(e_i)x_i^p + \ldots + \alpha_m\sum_{i=1}^n g_m(e_i)x_i^p = \sum_{i=1}^n f(e_i)x_i^p,
		\]
		which implies, 
			\begin{eqnarray*}
			\alpha_1\sum_{i=1}^n a_i^1 x_i^p + \alpha_2\sum_{i=1}^n a_i^2 x_i^p + \ldots + \alpha_m\sum_{i=1}^n a_i^m x_i^p = \sum_{i=1}^n b_i x_i^p,
		\end{eqnarray*}
		for any $p \in \{ 1, 2, \ldots, q\}.$
		This completes the necessary part of the theorem.
		
		\smallskip

		We now prove the sufficient part of the theorem. From the hypothesis, we obtain that $\alpha_1, \alpha_2, \ldots, \alpha_m\in \mathbb{R}$ satisfy the following relations:
		\begin{eqnarray}\label{x_t}
			\alpha_1\sum_{i=1}^n g_1(e_i)x_i^t + \alpha_2\sum_{i=1}^n g_2(e_i)x_i^t + \ldots + \alpha_1\sum_{i=1}^n g_m(e_i)x_i^t = \sum_{i=1}^n f(e_i)x_i^t,
		\end{eqnarray}
		for any $t \in \{ 1, 2, \ldots , q\}.$ 
		Now \[ (f - \sum_{k=1}^m\alpha_kg_k)\widetilde{x_t} = f(\widetilde{x_t})- \bigg\{ \alpha_1g_1(\widetilde{x_t}) + \alpha_2g_2(\widetilde{x_t}) +\ldots + \alpha_mg_m(\widetilde{x_t})   \bigg\}. \]
		Therefore, using equation (\ref{x_t}), it is immediate that  $\widetilde{x_t} \in ker(f - \sum_{k=1}^m\alpha_kg_k),$ for any $t \in \{1, 2, \ldots, q\}.$ 
		For any $\beta_1, \beta_2, \ldots, \beta_m \in \mathbb{R},$ not all zero,  there exists $ s \in \{ 1, 2, \ldots, q\}$ such that $\widetilde{x_s} \in M_{\sum_{k=1}^m\beta_kg_k}.$    Therefore, \[\widetilde{x_s} \in ker(f - \sum_{k=1}^m\alpha_kg_k) \cap  M_{\sum_{k=1}^m\beta_kg_k}. \]
		Therefore, from Theorem \ref{general}, the sufficient part of the theorem follows directly. 
		
	\end{proof}
	
	Combining Theorem \ref{characterization} with  the theoretical characterization of best coapproximation in terms of norm one projections, as given in \cite{LT, PS}, we get the following result.

	\begin{cor}
			Let $\mathbb{{Y}}  $ be a subspace of $ \ell_1^n $  and let $ \mathcal{A} = \{  \widetilde{a_1}, \widetilde{a_2}, \ldots, \widetilde{a_m} \}  $  be a basis of $\mathbb{{Y}}  $ with $\mathcal{Z}_{\mathcal{A}} = \emptyset.$ 
		Suppose that $ \mathcal{N} = \{ \pm \widetilde{x_1}, \pm \widetilde{x_2},\ldots, \pm \widetilde{x_q} \}$ is  the minimal norming set of $\mathbb{{Y}},$ where $\widetilde{x_k} = (x_1^k, x_2^k,  \ldots, x_n^k),$ for any $1\leq k\leq q.$ Then given $ \widetilde{b}=(b_1, b_2, \ldots, b_n) \in \ell_1^n,$ there exists a norm one projection from $span\{  \widetilde{b}, \mathbb{{Y}} \}$ to $\mathbb{{Y}}$ if and only if there exist 
		 $\alpha_1, \alpha_2, \ldots, \alpha_m \in \mathbb{R}$ satisfy the following relations:
		\begin{eqnarray*}
			\alpha_1\sum_{i=1}^n a_i^1 x_i^p + \alpha_2\sum_{i=1}^n a_i^2 x_i^p + \ldots + \alpha_m\sum_{i=1}^n a_i^m x_i^p = \sum_{i=1}^n b_i x_i^p,
		\end{eqnarray*}
		for any $p \in \{ 1, 2, \ldots, q\}.$ Moreover, if $\alpha_1, \alpha_2, \ldots, \alpha_m \in \mathbb{R}$ satisfy the above system of linear equations then  $P(\widetilde{a}+\gamma \widetilde{b}) = \widetilde{a}+ \gamma(\sum_{i=1}^{m} \alpha_i \widetilde{a_i})$ is the norm $1$ projection, for any $\widetilde{a} \in \mathbb{{Y}}.$
	\end{cor}

	We next obtain an immediate corollary from Theorem \ref{characterization}, which guarantees the uniqueness of best coapproximation to a element in $\ell_1^n$ out of a subspace of $\ell_1^n,$ provided it exists.
	
	\begin{cor}\label{cop}
		Let $\mathbb{{Y}}  $ be a subspace of $ \ell_1^n $  and let $ \mathcal{A} = \{  \widetilde{a_1}, \widetilde{a_2}, \ldots, \widetilde{a_m} \}  $  be a basis of $\mathbb{{Y}}  $ with $\mathcal{Z}_{\mathcal{A}} = \emptyset.$   For any given $f \in \ell_1^n,$ if there exists a best coapproximation to $f$ out of $\mathbb{{Y}}$ then it is unique.
	\end{cor}

	\begin{proof} 	Let $\{ \pm \widetilde{x_1}, \pm \widetilde{x_2},\ldots, \pm \widetilde{x_q} \}$ be the minimal  norming set of $\mathbb{{Y}},$ where $\widetilde{x_k} = (x_1^k, x_2^k, $ $ \ldots, x_n^k),$ for each $1\leq k\leq q.$
		Suppose on the contrary, $\sum_{k=1}^m\alpha_k \widetilde{a_k}$ and $\sum_{k=1}^m\gamma_k \widetilde{a_k}$ are  two distinct best coapproximations to $f$ out of $\mathbb{Y}.$ Therefore, from Theorem $\ref{characterization},$ $ \widetilde{\alpha}=(\alpha_1, \alpha_2, \ldots, \alpha_m) \in \mathbb{R}^m$ and $\widetilde{ \gamma} = (\gamma_1, \gamma_2, \ldots, \gamma_m) \in \mathbb{R}^m$ such that $ \alpha_i \neq \gamma_i$, for some $i \in \{1, 2, \ldots, m\}$ satisfies the following relations :\\
		\begin{eqnarray*}
			\alpha_1\sum_{i=1}^n a_i^1 x_i^p + \alpha_2\sum_{i=1}^n a_i^2 x_i^p + \ldots + \alpha_m\sum_{i=1}^n a_i^m x_i^p = \sum_{i=1}^n b_i x_i^p
		\end{eqnarray*}
		and
		\begin{eqnarray*}
			\gamma_1\sum_{i=1}^n a_i^1 x_i^p + \gamma_2\sum_{i=1}^n a_i^2 x_i^p + \ldots + \gamma_m\sum_{i=1}^n a_i^m x_i^p = \sum_{i=1}^n b_i x_i^p.
		\end{eqnarray*}
		for every $p \in \{1, 2, \ldots, q\}.$
		It is immediate from the above two equations that
		\begin{eqnarray*}
			(\alpha_1-\gamma_1)\sum_{i=1}^n a_i^1 x_i^p + (\alpha_2 - \gamma_2)\sum_{i=1}^n a_i^2 x_i^p + \ldots + (\alpha_m - \gamma_m)\sum_{i=1}^n a_i^m x_i^p = 0, 
		\end{eqnarray*}
		for all $p \in \{1, 2, \ldots, q\}.$
		Again using Theorem \ref{characterization}, we conclude that $\sum_{k=1}^m(\alpha_k - \gamma_k)\widetilde{a_k}$ is a best coapproximation to $\theta \in (\ell_{\infty}^n)^*$ out of $\mathbb{Y}.$ Therefore, $\alpha_k = \gamma_k,$ for all $k \in \{1, 2, \ldots, m\}.$ This contradiction completes the proof.
	\end{proof}

	Following   Corollary \ref{cop}, it is immediate that any coproximinal subspace $ \mathbb{{Y}}$ of $ \ell_1^n$ is also a co-Chebyshev subspace. We are now going to characterize the coproximinal(co-Chebyshev) subspaces with the help of Theorem \ref{characterization}. We first prove the following proposition.

		\begin{prop}\label{number}
		Let $\mathbb{{Y}}  $ be a subspace of $ \ell_1^n $  and let $ \mathcal{A} = \{  \widetilde{a_1}, \widetilde{a_2}, \ldots, \widetilde{a_m} \}  $  be a basis of $\mathbb{{Y}}  $ with $\mathcal{Z}_{\mathcal{A}} = \emptyset.$  Suppose that there are $ d$ number of nonequivalent components. Then there exist at most $d $ number of linearly independent elements in the minimal  norming set of $ \mathbb{Y}.$ 
	\end{prop}
	
	\begin{proof}
		Let $\{ \pm \widetilde{x_1}, \pm \widetilde{x_2},\ldots, \pm \widetilde{x_q} \}$ be the minimal norming set of $\mathbb{{Y}},$ where $\widetilde{x_k} = (x_1^k, x_2^k, $ $ \ldots, x_n^k),$ for each $1\leq k\leq q.$
		Let $ U = (u_{ij})_{ 1 \leq i \leq q, 1\leq j \leq n}$ such that $ u_{ij}= x_j^i.$  
		If the $ r$-th position and the $ s$-th position are equivalent then from the description of $\widetilde{x_k},$ defined in the Theorem \ref{norming:set} it is easy to observe that  $(x_r^1, x_r^2, \ldots, x_r^m) = \pm (x_s^1, x_s^2, \ldots, x_s^m). $ 
		Therefore, it is easy to observe that $ rank(U) \leq d.$ In other words,  there exist at most $d $ number of linearly independent elements in the norming set of $ \mathbb{Y}.$ 
	\end{proof}

	\begin{theorem}\label{coproxi}
		Let $\mathbb{{Y}}  $ be a subspace of $ \ell_1^n $  and let $ \mathcal{A} = \{  \widetilde{a_1}, \widetilde{a_2}, \ldots, \widetilde{a_m} \}  $  be a basis of $\mathbb{{Y}}  $ with $\mathcal{Z}_{\mathcal{A}} = \emptyset.$ 	Suppose that $\mathcal{N}$ is a minimal  norming set of $\mathbb{{Y}}$ and $\dim(span~ \mathcal{N}) = q.$ Then $ \mathbb{{Y}}$ is a coproximinal subspace if and only if $q=m.$
	\end{theorem}
	
	\begin{proof} 	Suppose that $\{ \widetilde{x_1}, \widetilde{x_2}, \ldots, \widetilde{x_q}\} $ is a linearly independent set in $\mathcal{N},$ where $ \widetilde{x_k}=(x_1^k, x_2^k, \ldots, x_n^k),$ where $ 1 \leq k \leq q.$	Let $T \in	\mathbb{L}(\mathbb{R}^m, \mathbb{R}^q)$ be a linear operator defined by 
		\[
		T(\widetilde{\alpha})= \bigg( 	\sum_{j=1}^{m} \alpha_j ( \sum_{i=1}^{n} a_i^j x_i^1), \sum_{j=1}^{m} \alpha_j ( \sum_{i=1}^{n} a_i^j x_i^2), \ldots, \sum_{j=1}^{m} \alpha_j (  \sum_{i=1}^{n} a_i^j x_i^q) \bigg),
		\]
		where $\widetilde{\alpha}= (\alpha_1, \alpha_2, \ldots, \alpha_m) \in \mathbb{R}^m.$
		Whenever $T(\alpha_1, \alpha_2, \ldots, \alpha_m)=0$ then it is easy to observe that  $ \sum_{k=1}^m \alpha_k \widetilde{a_k}$ is the best coapproximation to $\theta \in (\ell_{\infty}^n)^*$ out of $ \mathbb{{Y}}.$ Clearly, $  (\alpha_1, \alpha_2, \ldots, \alpha_m) = \theta \in \mathbb{R}^m$ and therefore, $ ker~T= \theta \in  \mathbb{R}^m. $ In other words, $ q \geq m.$
		To prove the necessary part of the theorem we only need to show $q \leq m.$
		Let us now take $(u_1, u_2, \ldots, u_q)\in \mathbb{R}^q.$
		 Then we choose $ \widetilde{b}=(b_1, b_2, \ldots, b_n) \in \ell_1^n $ such that $ \sum_{i=1}^{n} b_i x_i^p= u_p,$ for any $1 \leq p \leq q.$
		   Since $\{ \widetilde{x_1}, \widetilde{x_2}, \ldots, \widetilde{x_q}\} $ is  linearly independent, the existence of such $\widetilde{b}$ is always guaranteed.
		As $ \mathbb{{Y}}$ is coproximinal, following Theorem \ref{characterization} we obtain that for any $ \widetilde{b} \in \ell_1^n,$  there exists $ \widetilde{\alpha}=( \alpha_1, \alpha_2, \ldots, \alpha_m) \in \mathbb{R}^m$ satisfying  
		\begin{eqnarray*}
			\alpha_1\sum_{i=1}^n a_i^1 x_i^p + \alpha_2\sum_{i=1}^n a_i^2 x_i^p + \ldots + \alpha_m\sum_{i=1}^n a_i^m x_i^p = \sum_{i=1}^n b_i x_i^p,
		\end{eqnarray*}
		for any $p \in \{ 1, 2, \ldots, q\}.$  Therefore, $T(\widetilde{\alpha})= (u_1, u_2, \ldots, u_q), $ which implies that $T$ is onto.
		Consequently, $ q \leq m,$ establishing the necessary part of the theorem.
		
		\smallskip

		Let us now prove the sufficient part of the theorem.  Since $ q=m$ and $ker~T= \theta \in \mathbb{R}^m,$ it is immediate that $ T$ is invertible.   This implies that for any $ \widetilde{b} \in \ell_1^n,$ there exists $\widetilde{\alpha}=(\alpha_1, \alpha_2, \ldots, \alpha_m) \in \mathbb{R}^m$ satisfying 
		\[
		T(\widetilde{\alpha})= \bigg(\sum_{i=1}^n b_i x_i^1, \sum_{i=1}^n b_i x_i^2, \ldots, \sum_{i=1}^n b_i x_i^q\bigg),\]
		which implies 
		\begin{eqnarray*}
			\alpha_1\sum_{i=1}^n a_i^1 x_i^p + \alpha_2\sum_{i=1}^n a_i^2 x_i^p + \ldots + \alpha_m\sum_{i=1}^n a_i^m x_i^p = \sum_{i=1}^n b_i x_i^p,
		\end{eqnarray*}
		for every $p \in \{ 1, 2, \ldots, q\}.$ Since $\{ \widetilde{x_1}, \widetilde{x_2}, \ldots, \widetilde{x_q}\}$ is a basis of $span~ \mathcal{N}$, using Theorem \ref{characterization} we conclude that $\mathbb{{Y}}$ is a coproximinal subspace of $\ell_1^n.$
	\end{proof}

	We now present an explicit numerical example to illustrate the applicability of Theorem \ref{characterization} towards solving the best coapproximation problem in $ \ell_1^n, $ from a  computational point of view.

\begin{example}\label{example}
	Find the best coapproximation(s) to any given $\widetilde{b} \in \ell_1^6$ out of the subspace $ \mathbb{{Y}} = span\left\lbrace \widetilde{a_1}, \widetilde{a_2}, \widetilde{a_3} \right\rbrace$ of  $  \ell_1^6,$ where 
	$ \widetilde{a_1} = ( ~4, ~2,~ 1,~ -1,~ -4,~ 4),~ \widetilde{a_2} = (~ -1,~ 3, ~5, ~2, ~1,~6),~ \widetilde{a_3} = (~ 1, ~4, ~2, ~1, ~-1, ~8) \in \ell_1^6.$\\
	
	\textbf{Step 1:}
	Let $\psi(\widetilde{a_i}) = g_i \in (\ell_\infty^n)^*,$ where $\psi$ is the canonical isometric isomorphism from $\ell_1^6$ to $(\ell_\infty^6)^*. $ 
	Here, for any $(x_1, x_2, \ldots, x_6) \in \ell_\infty^6,$
	\[
	g_1(x_1, x_2, x_3, x_4, x_6) = 4x_1+ 2x_2+ x_3-x_4- 4x_5 + 4x_6,\]
	\[
	g_2(x_1, x_2, x_3, x_4, x_6) = -x_1+ 3x_2+ 5x_3+2x_4 +x_5 +6x_6,\]
	\[
	g_3(x_1, x_2, x_3, x_4, x_6) = x_1+ 4x_2+ 2x_3+x_4- x_5 + 8x_6.\]
	We first observe that $\mathcal{Z}_{\mathcal{A}} = \emptyset.$ From Theorem  \ref{norming:set}, suppose that $\mathcal{N}$ is the unique minimal norming set of $\mathbb{{Y}}.$\\
	
	\textbf{Step  2:} We observe that the $1$-st, $2$-nd, $3$-rd and $4$-th positions can be taken as the nonequivalent components.
	% Here	$ P_1^+ = \left\lbrace 1\right\rbrace , P_1^- = \left\lbrace 5\right\rbrace ; ~ P_2^+ = \left\lbrace 2, 6\right\rbrace , P_2^- = \phi ; ~ P_3^+ = \left\lbrace 3\right\rbrace , P_3^- = \phi; ~ P_4^+ = \left\lbrace 4\right\rbrace , P_4^- = \phi;  ~P_5^+ = \left\lbrace 5\right\rbrace , P_5^- = \left\lbrace1 \right\rbrace ; ~ P_6^+ = \left\lbrace 2,6\right\rbrace , P_6^- = \phi  $ , respectively, where $P_i^+$ and   $P_i^-$ are the positively associated set and the negatively associated set of the $i$-th position, respectively, for all $ i \in \left\lbrace 1, 2, \ldots ,6\right\rbrace.$
	 The  hyperplanes corresponding to each components are:
	\[ H_1 = \bigg\{ (\beta_1, \beta_2, \beta_3) \in \mathbb{R}^3 : 4\beta_1- \beta_2+ \beta_3= 0 \bigg\},\]
	\[	H_2 = \bigg\{ (\beta_1, \beta_2, \beta_3) \in \mathbb{R}^3 : 2\beta_1+ 3 \beta_2+ 4\beta_3= 0 \bigg\},\]
	\[	H_3 = \bigg\{ (\beta_1, \beta_2, \beta_3) \in \mathbb{R}^3 : \beta_1 +5\beta_2+ 2\beta_3= 0 \bigg\}, \]
	\[	H_4 = \bigg\{ (\beta_1, \beta_2, \beta_3) \in \mathbb{R}^3 : -\beta_1+ 2\beta_2+ \beta_3= 0 \bigg\}.	\]
	\[
	H_5=\bigg\{ (\beta_1, \beta_2, \beta_3) \in \mathbb{R}^3 : -4\beta_1+ \beta_2- \beta_3= 0 \bigg\},
	\]
	and 
	\[
	H_6=\bigg\{ (\beta_1, \beta_2, \beta_3) \in \mathbb{R}^3 : 4\beta_1+ 6 \beta_2+ 8 \beta_3= 0 \bigg\},
	\]
	Clearly, $H_5 = H_1, H_6 = H_2$ and   $H_5^+ = H_1^-,$ $H_5^- = H_1^+;$ $H_6^+= H_2^+, H_6^- = H_2^-.$\\

	\textbf{ Step  3:}
	To solve the best coapproximation problem with the help of Theorem \ref{characterization}, we first need to find a basis of $span~\mathcal{N}.$ We observe that  there are four nonequivalent positions, and therefore, from Proposition \ref{number} we note that $\dim(span ~\mathcal{N}) \leq 4.$
	
	As mentioned in  Theorem \ref{norming:set},  we consider the sets $K_i= H_1^{\delta_{i_1}} \cap H_2^{\delta_{i_2}} \cap H_3^{\delta_{i_3}} \cap H_4^{\delta_{i_4}},$ where $\delta_{i_j} \in \{ +, -\},$ for any $j \in \{1,2,3,4\}.$  Although there are $2^4$ number of possible $K_i$'s, it is evident that we only need to take account of the nonempty $K_i$'s.  %interested on those sets which are nonempty. 
	Moreover, we  associate each of these nonempty $K_i$'s with an extreme point $\widetilde{x_i}$ of $B_{\ell_\infty^6},$ as mentioned in   Theorem \ref{norming:set}.
	
	Suppose that $K_1= H_1^+ \cap H_2^+ \cap H_3^+ \cap H_4^+$ and it is straightforward to verify that $(1, 2, 3) \in K_1.$ Therefore, we obtain $\widetilde{x_1}= (1, 1, 1, 1, -1, 1) \in Ext(B_{\ell_\infty^6}).$ 
	
	In a similar manner, we take 
	\[K_2= H_1^+ \cap H_2^+ \cap H_3^+ \cap H_4^-,  \quad  ( 4, -1 , 1 ) \in K_2.\]
	Therefore, $\widetilde{x_2}= ( 1, 1, 1, -1, -1, 1).$
	Again 
	\[
	K_3 =  H_1^+ \cap H_2^+ \cap H_3^- \cap H_4^-, \quad ( 0, -1, \frac{3}{2}) \in K_3.
	\]
	So, $ \widetilde{x_3}= (1, 1, -1, -1, -1, 1).$
	Also take
	\[
	K_4=  H_1^+ \cap H_2^- \cap H_3^- \cap H_4^-, \quad (1, 0, -1) \in K_4.
	\]
	Therefore, we have $\widetilde{x_4}= ( 1, -1, -1, -1, -1, -1).$

	From  Theorem \ref{norming:set}, it is now immediate that $ \widetilde{x_1}, \widetilde{x_2}, \widetilde{x_3}, \widetilde{x_4} \in \mathcal{N}.$
	It is straightforward to check that $\{ \widetilde{x_1}, \widetilde{x_2}, \widetilde{x_3}, \widetilde{x_4}\} $ is  linearly independent. Therefore,
	$\{ \pm  ( 1, 1, 1, 1, \\ -1,  1), \pm ( 1, 1, 1, -1, -1, 1), \pm ( 1, 1,- 1, -1, -1, 1), \pm  ( 1, -1,- $ $ 1, -1, -1, -1)  \}
	$ 
	is a basis  of $span~\mathcal{N}$.\\

	\textbf{ Step  4:} In this final step, by considering a given $ \widetilde{b} \in \ell_1^6$  and thereafter applying Theorem \ref{characterization}, we  obtain the best coapproximation to $ \widetilde{b} $ out of $ \mathbb{{Y}}. $ In order to illustrate the various possibilities arising in the best coapproximation problem in $ \ell_1^6, $ it suffices to consider the following two particular cases.
	
	$Case~1:$
	Let  $\widetilde{b_1} = \left( 1, 2, 3, 4, 5, 6\right) \in \ell_1^6.$ 
		Then from Theorem \ref{characterization}, $\sum_{i=1}^{3}\alpha_i g_i $ is a best coapproximation to $ \widetilde{b}_1$ out of $ \mathbb{{Y}}$ if and only if $ \alpha_1, \alpha_2,\alpha_3 \in \mathbb{R}$ satisfies the following relations:
	\begin{eqnarray*}
		14\alpha_1 + ~ 14\alpha_2 + ~ 17\alpha_3  &=& 11 
		\\ 16\alpha_1 +~ 10\alpha_2 +~15\alpha_3 &=& 3
		\\	14\alpha_1 + \space 11\alpha_3 &=& -3
		\\	2\alpha_1 -~18\alpha_2 - ~13\alpha_3 &=& -19.
	\end{eqnarray*}
	
	Since there exist no such  $\alpha_1 ,\alpha_2 ,\alpha_3 \in \mathbb{R} $ satisfying the above relations, it follows that
	\[ \mathcal{R}_\mathbb{{Y}}(\widetilde{b_1}) = \phi. \]\\
	
	$Case~2:$ 	Let $ \widetilde{b_2} = ( 5, 4, 0, 0, 1, 5) \in \ell_1^6.$ 
		Then from Theorem \ref{characterization}, $\sum_{i=1}^{3}\alpha_i\widetilde{a_i} $ is a best coapproximation to $ \widetilde{b}_2$ out of $ \mathbb{{Y}}$ if and only if $ \alpha_1, \alpha_2,\alpha_3 \in \mathbb{R}$ satisfies the following relations:
	\begin{eqnarray*}
		14\alpha_1 + ~ 14\alpha_2 + ~ 17\alpha_3  &=& 13 
		\\ 16\alpha_1 +~ 10\alpha_2 +~15\alpha_3 &=& 13
		\\	14\alpha_1  \quad \quad  \quad \quad + 11\alpha_3 &=& 13
		\\	2\alpha_1 -~18\alpha_2 - ~13\alpha_3 &=& -5.
	\end{eqnarray*}
	
	Since there exist unique  $\alpha_1 ,\alpha_2 ,\alpha_3 \in \mathbb{R} $ satisfying the above relations, the best coapproximation to $\widetilde{b_2}$ out of $ \mathbb{{Y}}  $ is unique.  Moreover, $\alpha_1= \frac{1}{7}, \alpha_2=- \frac{3}{7}, \alpha_3=1$ and therefore
	\[
	\mathcal{R}_\mathbb{{Y}}(\widetilde{b_2}) = \frac{1}{7} \widetilde{a_1} - \frac{3}{7} \widetilde{a_2} + \widetilde{a_3}= (2, 3, 0, 0, -2, 6).
	\]
\end{example}

	We end this section with the following remark.
	
	\begin{remark}
		It is already known from   \cite{LT, PS} that  whenever the best coapproximation exists from a element $\widetilde{b}$ to a subspace $\mathbb{{Y}}$ of $\ell_1^n,$ then  there exist a norm one projection from $span\{\widetilde{b}, \mathbb{{Y}}\}$ to $\mathbb{{Y}}.$ However, it is also natural to look for the explicit description of the concerned projection. In view of the method developed here, we can now find the projection map explicitly. If we consider the subspace $\mathbb{Y}$ and element $ \widetilde{b_2}$ as in Example \ref{example}, then the norm one  projection $P$ from $span\{  \widetilde{b_2}, \mathbb{{Y}}\}$ to $\mathbb{{Y}}$ is given by: 
		\[  
		P(\widetilde{b_2}) = P(5, 4, 0, 0, 1, 5)=(2, 3, 0, 0, -2, 6) ; \quad \quad P(y)=y, \forall y \in \mathbb{{Y}} .
	\]\\
	\end{remark}
	
In the next section, we deal with subspaces $\mathbb{Y}$  for which the zero set $ \mathcal{Z}_{\mathcal{A}} \neq \emptyset,$ where $\mathcal{A}$ is a basis of $\mathbb{Y}.$

		\section*{Section-II}	
	
	We note  that the uniqueness of minimal norming set of $\mathbb{Y}$ plays a pivotal role in obtaining the complete characterization of the best coapproximation problem, when the zero set is empty.  On the other hand, whenever $\mathcal{Z}_{\mathcal{A}} \neq \emptyset,$ the subspace $\mathbb{Y}$ does not possess this property. However, a sufficient condition is clearly possible by choosing a minimal norming set of $\mathbb{Y}.$  We  tackle the problem in this section with a different technique, more precisely with the help of the norm of the space $ \ell_1^n,$ to obtain a complete characterization of the problem.  The following definitions and notations are needed  throughout this section to complete the desired characterization.

	\begin{definition}
		Let  $\mathbb{Y}$ be a subspace of $\ell_1^n$ and let $ \mathcal{A}= \{\widetilde{a_1}, \widetilde{a_2}, \ldots, \widetilde{a_m}\}$ be a basis of $\mathbb{Y}$ with $\mathcal{Z}_{\mathcal{A}} \neq \emptyset.$   Suppose that $ |\mathcal{Z}_{\mathcal{A}}|= r>0$ and $n-r = k.$ Without loss of generality we assume that $ \{ 1, 2, \ldots, n\} \setminus \mathcal{Z}_{\mathcal{A}}=\{ 1, 2, \ldots, k\}.$
		
		\begin{itemize}
			\item[(i)] We define a linear transformation  $\rho$ from $\ell_1^n$ to $\ell_1^n$ by 
			\[
			\rho(b_1, b_2, \ldots, b_n)=(c_1, c_2, \ldots, c_n),\] 
			where $ c_i= b_i,$ for any $i \notin \mathcal{Z}_{\mathcal{A}}$ and $c_i=0,$ for any $i \in \mathcal{Z}_{\mathcal{A}}.$	\\

			\item[(ii)]  We define a linear transformation  $\sigma $ from $\ell_1^n$ to $ \ell_1^k$ by 
			\[
			\sigma(b_1, b_2, \ldots, b_n) = (b_1, b_2, \ldots, b_k).
			\] 
			
			\item[(iii)]  For a given $\widetilde{b}=(b_1, b_2, \ldots, b_n) \in \ell_1^n,$ we introduce a set $ \mathcal{P}_{\widetilde{b}} \subset \ell_1^n,$ defined as  
			\[ \mathcal{P}_{\widetilde{b}}:  = \bigg\{ \widetilde{y}=(y_1, y_2, \ldots, y_n) \in \ell_1^n : y_i=b_i ~\forall j \notin \mathcal{Z}_{\mathcal{A}} \bigg\}.\]
			
		\end{itemize}
	\end{definition}

	We note the following  simple but useful properties in the form of a  proposition. 
	\begin{prop}	\label{prop} 
		Let  $\mathbb{Y}$ be a subspace of $\ell_1^n$ and let $ \mathcal{A}= \{\widetilde{a_1}, \widetilde{a_2}, \ldots, \widetilde{a_m}\}$ be a basis of $\mathbb{Y}$ with $\mathcal{Z}_{\mathcal{A}} \neq \emptyset.$   Suppose that $ |\mathcal{Z}_{\mathcal{A}}|= r>0$ and $n-r = k.$ 	For any $f \in (\ell_\infty^n)^*,$
		\begin{itemize}
			\item[(i)] $\rho(\widetilde{a_i})=\widetilde{a_i}$ and $\| \rho(\widetilde{b})\|\leq \|\widetilde{b}\|.$ 
			\item [(ii)] $\|\sigma(\widetilde{b})\| \leq \|\widetilde{b}\|$ and $\|\sigma(\widetilde{a_i})\|=\|\widetilde{a_i}\|.$
			\item [(iii)] $\sigma(\rho(\widetilde{b}))= \sigma(\widetilde{b})$ and $\rho(\rho(\widetilde{b}))= \rho(\widetilde{b}).$
			\item [(iv)] for any $ \widetilde{y} \in \mathcal{P}_{\widetilde{b}},  ~\rho(\widetilde{b})= \rho(\widetilde{y}).$
			\item[(v)] $\mathcal{Z}_{\sigma(\mathcal{A})} = \phi,$ where $ \sigma(\mathcal{A})= \{ \sigma(\widetilde{a_1}), \sigma(\widetilde{a_2}), \ldots, \sigma(\widetilde{a_m})\}.$
		\end{itemize}
	\end{prop}

	In the following theorem, a sufficient condition for the best coapproximation to an element $f$ out of a subspace $\mathbb{Y}$  is given.
	
	\begin{theorem}\label{sufficient}
		Let  $\mathbb{Y}$ be a subspace of $\ell_1^n$ and let $ \mathcal{A}= \{\widetilde{a_1}, \widetilde{a_2}, \ldots, \widetilde{a_m}\}$ be a basis of $\mathbb{Y}$ with $|\mathcal{Z}_{\mathcal{A}}| = r > 0.$   Then for any $\widetilde{b} \in \ell_1^n, $ $\sum_{i=1}^{m} \alpha_i \widetilde{a_i}$ is a best coapproximation to $\widetilde{b}$ out of $\mathbb{{Y}}$ if $\sum_{i=1}^{m} \alpha_i \sigma (\widetilde {a_i})$ is a best coapproximation to $ \sigma(\widetilde{b})$ out of  $ \sigma(\mathbb{Y}),$ where  $ \sigma(\mathbb{Y}) = span\{\sigma(\widetilde{a_1}), \sigma(\widetilde{a_2}), \ldots, \sigma(\widetilde{a_m})\}. $
	\end{theorem}
	
	\begin{proof}
		Since $\sum_{i=1}^{m} \alpha_i \sigma(\widetilde{a_i})$ is a best coapproximation to $ \sigma(\widetilde{b})$ out of $ \sigma(\mathbb{Y}),$  so	for any $ \beta_1, \beta_2, \ldots, \beta_m \in \mathbb{R},$ we have
		\begin{eqnarray*}
			\| \sum_{i=1}^{m}\beta_i \widetilde{a_i} - \sum_{i=1}^{m}\alpha_i \widetilde{a_i} \| =  \| \sum_{i=1}^{m}\beta_i \sigma(\widetilde{a_i}) - \sum_{i=1}^{m}\alpha_i \sigma(\widetilde{a_i})\|   &\leq&  \| \sum_{i=1}^{m}\beta_i \sigma(\widetilde{a_i}) - \sigma(\widetilde{b}) \| \\ &=& \|\sigma( \sum_{i=1}^{m} \beta_i \widetilde{a_i} - b)\| \\
			&\leq&  \|\sum_{i=1}^{m} \beta_i \widetilde{a_i} - \widetilde{b}\|. 
		\end{eqnarray*}
		In other words, $\sum_{i=1}^{m} \alpha_i \widetilde{a_i}$ is a best coapproximation to $ \widetilde{b} $ out of $\mathbb{{Y}}.$ This establishes our theorem.
	\end{proof}

\begin{remark}
	Observe that the zero set  corresponding to a basis of $ \sigma(\mathbb{Y})$ is empty and so following the method discussed in Section I, we can find the best coaproximation    to $ \sigma(\widetilde{b})$ out of  $ \sigma(\mathbb{Y}),$ which in turn allows us to exactly find the best coapproximation to $\widetilde{b}$ out of $\mathbb{{Y}}.$
\end{remark}
	
	In order to characterize the best coapproximation problem in $ \ell_1^n, $ we require the following lemma.

	\begin{lemma}\label{lemma1}
		Let  $\mathbb{Y}$ be a subspace of $\ell_1^n$ and let $ \mathcal{A}= \{\widetilde{a_1}, \widetilde{a_2}, \ldots, \widetilde{a_m}\}$ be a basis of $\mathbb{Y}$ with $|\mathcal{Z}_{\mathcal{A}}| = r > 0.$   Suppose that for $\widetilde{b} \in \ell_1^n, $ there exists no best coapproximation to $ \sigma(\widetilde{b})$ out of $ \sigma(\mathbb{Y}) = span \{ \sigma(\widetilde{a_1}), \sigma(\widetilde{a_2}), \ldots, \sigma(\widetilde{a_m}) \}.$ Then there exists  $ \delta > 0$ such that for any $ y \in \mathcal{B}_\delta(\rho(\widetilde{b})) \bigcap~ \mathcal{P}_{\widetilde{b}} ,$ there exists no best coapproximation to $ y $ out of $\mathbb{Y}. $
	\end{lemma}
	
	\begin{proof}
		Since there exists no best coapproximation to $ \sigma(\widetilde{b}) $ out of  $ \sigma(\mathbb{Y}),$  then  for any $ (\alpha_1, \alpha_2, \ldots, \alpha_m ) \in \mathbb{R}^m,$ there exists $ (\beta_1, \beta_2, \ldots, \beta_m ) \in \mathbb{R}^m$ such that
		\[ \| \sum_{i=1}^{m}\beta_i\sigma(\widetilde{a_i}) - \sum_{i=1}^{m}\alpha_i\sigma(\widetilde{a_i})\|~ >
		~ \| \sum_{i=1}^{m}\beta_i\sigma(\widetilde{a_i}) - \sigma(\widetilde{b}) \| = \| \sigma ( \sum_{i=1}^{m} \beta_i \widetilde{a_i} - \widetilde{b})\|.\]
		It is easy to observe that for any $\widetilde{u} \in \ell_1^n  ,$  if $\rho(\widetilde{u})= \widetilde{u}$ then $\| \sigma(\widetilde{u})\| = \|\widetilde{u}\|.$ Since $\rho\bigg(\sum_{i=1}^{m} \beta_i \widetilde{a_i} - \rho(\widetilde{b}) \bigg)=\sum_{i=1}^{m}\beta_i \widetilde{a_i} - \rho(\widetilde{b}), $ we have
		\[
		\| \sigma \bigg( \sum_{i=1}^{m} \beta_i \widetilde{a_i} - \widetilde{b} \bigg)\| = \| \sigma \bigg(\sum_{i=1}^{m} \beta_i \widetilde{a_i} - \rho(\widetilde{b}) \bigg)\|= \| \sum_{i=1}^{m}\beta_i \widetilde{a_i} - \rho(\widetilde{b}) \|
		\]
		and therefore,
		\[
		\| \sum_{i=1}^{m}\beta_i \widetilde{a_i} - \sum_{i=1}^{m}\alpha_i \widetilde{a_1}\| = \| \sum_{i=1}^{m}\beta_i\sigma(\widetilde{a_i}) - \sum_{i=1}^{m}\alpha_i \sigma(\widetilde{a_i})\| > \| \sum_{i=1}^{m}\beta_i \widetilde{a_i} - \rho(\widetilde{b})\|.
		\]
		In other words,\ there exists no best coapproximation to $ \rho(\widetilde{b}) $ out of $  \mathbb{Y}. $ Since for any $ \widetilde{u} \in \ell_1^n,$ $\mathcal{R}_{\mathbb{Y}}(\widetilde{u})$ is a compact set \cite{N}, it is immediate that $\mathcal{D}(\mathcal{R}_{\mathbb{Y}})$ is  closed. Therefore, there exists an open ball of radius $\delta > 0 $ (say) centered at $ \rho(\widetilde{b}), ~ \mathcal{B}_\delta(\rho(\widetilde{b}))$ such that for any $ y \in \mathcal{B}_\delta (\rho(\widetilde{b})) \bigcap \mathcal{P}_{\widetilde{b}}, $ there exists no best coapproximation to $ y $ out of $\mathbb{Y}. $ This proves our lemma.
		
	\end{proof}
	
	We next characterize the existence of the best coapproximation(s) in $ \ell_1^n $ in the following theorem.

	\begin{theorem}
		Let  $\mathbb{Y}$ be a subspace of $\ell_1^n$ and let $ \mathcal{A}= \{ \widetilde{a_1}, \widetilde{a_2}, \ldots, \widetilde{a_m}\}$ be a basis of $\mathbb{Y}$ with $|\mathcal{Z}_{\mathcal{A}}| = r > 0.$ Suppose that for $ \widetilde{b} \in \ell_1^n, $ there exists no best coapproximation to $ \sigma(\widetilde{b})$ out of $ \sigma(\mathbb{Y}) = span \{ \sigma(\widetilde{a_1}), \sigma(\widetilde{a_2}), \ldots, \sigma(\widetilde{a_m}) \}.$ Then there exists  $ \delta _0(>0) $ satisfying the following:
		\begin{itemize}
			\item[(i)]  for any $ y \in \mathcal{P}_{\widetilde{b}} $ 
			such that $ \| y-\rho(\widetilde{b}) \| < \delta_0,$
			there exists no best coapproximation to $ y $ out of $ \mathbb{{Y}},$
			
			\item[(ii)] for any $ y \in  \mathcal{P}_{\widetilde{b}} $ such that $ \|y-\rho(\widetilde{b})\| \geq \delta_0,$ there exists a best coapproximation to $ y $ out of $ \mathbb{{Y}}.$
		\end{itemize}
	\end{theorem}
	
	\begin{proof}
		Let us define the set 
		\[ 
		S:= \bigg\{ \delta \in \mathbb{R}:  \mathcal{R}_{\mathbb{Y}}(y) = \emptyset ~ \forall ~  y \in \mathcal{B}_\delta(\rho(\widetilde{b})) \bigcap \mathcal{P}_{\widetilde{b}} \bigg\}.
		\]
		To prove the theorem, we only need to show that $S$ has a upper bound.
		From Lemma \ref{lemma1}, it is assured that $ S \neq \emptyset$ and if there exists no best coapproximation to $ \sigma(\widetilde{b})$ out of $\sigma(\mathbb{Y})$ then there exists no best coapproximation to $ \rho(\widetilde{b})$ out of $ \mathbb{{Y}}.$
		It is easy to observe that for all  $ \beta_1, \beta_2, \ldots, \beta_m \in \mathbb{R},$
		\[
		\| \sum_{i=1}^{m}\beta_i \widetilde{a_i} \| \leq \|\sum_{i=1}^{m}\beta_i \widetilde{a_i} -\rho(\widetilde{b}) \| + \| \rho(\widetilde{b})\|.
		\]
		Suppose that $\widetilde{b}=(b_1, b_2, \ldots, b_n).$
		For any  $y \in \mathcal{P}_{\widetilde{b}}$, we observe that $ \rho(y)= \rho(\widetilde{b}) \in \ell_1^n,$ i.e., $y_i=b_i,$ for any $ i \notin \mathcal{Z}_{\mathcal{A}}.$
		We now choose $ y \in \mathcal{P}_{\widetilde{b}}$ such that $ \| y -\rho(\widetilde{b}) \| = \| \rho(\widetilde{b})\|$ then 
		\begin{equation}\label{neweq1}
			\| \sum_{i=1}^{m}\beta_i \widetilde{a_i} \| \leq \|\sum_{i=1}^{m}\beta_i \widetilde{a_i} -\rho(\widetilde{b}) \| + \|\rho(\widetilde{b})\|  = \| \sum_{i=1}^{m}\beta_i \widetilde{a_i} -\rho(\widetilde{b}) \| +\| y - \rho(\widetilde{b})\| 
		\end{equation}
		for all  $ \beta_1, \beta_2, \ldots, \beta_m \in \mathbb{R}.$
		Now we observe that 
		%\begin{eqnarray*}
		%	\| \sum_{i=1}^{m}\beta_i \widetilde{a_i} - y \| &=& \sum_{j=1}^{n} %\bigg| \bigg(  \sum_{i=1}^{m}\beta_i \widetilde{a_i} - y \bigg) (e_j ) \bigg|\\ &=& \sum_{j \notin \mathcal{Z}_{\mathcal{A}}} \bigg| \bigg(  \sum_{i=1}^{m}\beta_i g_i - g \bigg) (e_j) \bigg| + \sum_{j \in \mathcal{Z}_{\mathcal{A}}} \bigg| \bigg(  \sum_{i=1}^{m}\beta_ig_i - g \bigg) (e_j) \bigg|\\
		%	&=& \sum_{j \notin \mathcal{Z}_{\mathcal{A}}} \bigg| \bigg(  \sum_{i=1}^{m}\beta_i g_i - \rho(g) \bigg) (e_j) \bigg| + \sum_{j \in \mathcal{Z}_{\mathcal{A}}} |   g  (e_j) |\\
		%	&=& \sum_{j \notin \mathcal{Z}_{\mathcal{A}}} \bigg| \bigg(  \sum_{i=1}^{m}\beta_i g_i - \rho(f) \bigg) (e_j) \bigg| + \sum_{j \in \mathcal{Z}_{\mathcal{A}}} |   g ( e_j) |.
		%\end{eqnarray*}
		%Since
		%		\begin{eqnarray*}
		%		\sum_{j \in \mathcal{Z}_{\mathcal{A}}} \bigg| \bigg(  \sum_{i=1}^{m}\beta_i g_i - \rho(f) \bigg) (e_j) \bigg| = 0 =  \sum_{j \notin \mathcal{Z}_{\mathcal{A}}} | (g - \rho(f)) (e_j)|.
		%	\end{eqnarray*}
		\begin{eqnarray*}
			\| \sum_{i=1}^{m}\beta_i \widetilde{a_i} - y \| &=& \sum_{j=1}^{n} \bigg|  \bigg(  \sum_{i=1}^{m} \beta_i a_j^i-y_j  \bigg)  \bigg| \\
			& =& \sum_{j\notin \mathcal{Z}_{\mathcal{A}}} \bigg|  \bigg(  \sum_{i=1}^{m} \beta_i a_j^i-y_j  \bigg)  \bigg| + \sum_{j\in \mathcal{Z}_{\mathcal{A}}} \bigg|  \bigg(  \sum_{i=1}^{m} \beta_i a_j^i-y_j  \bigg)  \bigg|\\
			&=& \sum_{j\notin \mathcal{Z}_{\mathcal{A}}} \bigg|  \bigg(  \sum_{i=1}^{m} \beta_i a_j^i- b_j  \bigg)  \bigg| +  \sum_{j\in \mathcal{Z}_{\mathcal{A}}}|y_j|
			\end{eqnarray*}
			
	Therefore, we can easily obtain that 
			\begin{eqnarray}\label{neweqn2}	
			\| \sum_{i=1}^{m}\beta_i \widetilde{a_i} - y \|=	 \| \sum_{i=1}^{m}\beta_i \widetilde{a_i} -\rho(\widetilde{b}) \| +\| y - \rho(\widetilde{b})\|.
		\end{eqnarray}
		From equations (\ref{neweq1}) and (\ref{neweqn2}) we have 
		\[
		\| \sum_{i=1}^{m}\beta_i \widetilde{a_i} \| \leq 	\| \sum_{i=1}^{m}\beta_i \widetilde{a_i} - y \|.
		\]
		In other words,
		$ \theta=(0, 0, \ldots, 0) \in \ell_1^n$ is a best coapproximation to $ y $ out of $\mathbb{{Y}}.$ Take $\gamma > \| \rho(\widetilde{b})\|.$  Therefore, $ y \in \mathcal{B}_\gamma(\rho(\widetilde{b})) \bigcap \mathcal{P}_{\widetilde{b}}$ and consequently  $ \gamma $ is an upper bound of $ S.$ Let $ \sup S = \delta_0.$ Hence the theorem. 
	\end{proof}
	
	In  the following theorem  we characterize  the coproximinal subspace of $ \ell_1^n.$ 
	
	\begin{theorem}\label{coproximinal}
		Let  $\mathbb{Y}$ be a subspace of $\ell_1^n$ and let $ \mathcal{A}= \{ \widetilde{a_1}, \widetilde{a_2}, \ldots, \widetilde{a_m}\}$ be a basis of $\mathbb{Y}$ with $|\mathcal{Z}_{\mathcal{A}}| = r > 0.$ Let $n-r=k.$ Then $ \mathbb{{Y}} $ is a coproximinal subspace of $ \ell_1^n$ if and only if $ \sigma(\mathbb{Y}) = span\{\sigma(\widetilde{a_1}), \sigma(\widetilde{a_2}), \ldots, \sigma(\widetilde{a_m})\}$ is a coproximinal subspace of $ \ell_1^{k}.$
	\end{theorem}
	
	\begin{proof}
		Let us first prove the necessary part of the theorem.  For any $ \widetilde{w} \in \ell_1^k,$ we choose $ \widetilde{b} \in \ell_1^n$ such that $ \sigma(\widetilde{b})= \widetilde{w}.$ Since $ \mathbb{Y}$ is a coproximinal subspace of $ \ell_1^n,$ there exist  $ \alpha_1, \alpha_2, \ldots, \alpha_m \in \mathbb{R}$ such that  $ \sum_{i=1}^{m} \alpha_i \widetilde{a_i} $ is a best coapproximation to $ \rho(\widetilde{b})$ out of $ \mathbb{{Y}}.$ Therefore, for any $ \beta_1, \beta_2, \ldots, \beta_m \in \mathbb{R},$
		\begin{eqnarray*}
			\| \sum_{i=1}^{m} \beta_i \sigma(\widetilde{a_i}) - \sum_{i=1}^{m} \alpha_i \sigma(\widetilde{a_i}) \| = \|\sum_{i=1}^{m} \beta_i \widetilde{a_i} - \sum_{i=1}^{m} \alpha_i \widetilde{a_i}\|   \leq   \| \sum_{i=1}^{m} \beta_i \widetilde{a_i} - \rho(\widetilde{b}) \| 
		\end{eqnarray*}
		Now $ \rho \bigg( \sum_{i=1}^{m} \beta_i \widetilde{a_i} - \rho(\widetilde{b})\bigg)= \sum_{i=1}^{m} \beta_i \widetilde{a_i} - \rho(\widetilde{b}),$ so we have 
		\begin{eqnarray*}
			\|\bigg( \sum_{i=1}^{m} \beta_i \widetilde{a_i} - \rho(\widetilde{b})\bigg)\| =  \| \sigma \bigg( \sum_{i=1}^{m} \beta_i \widetilde{a_i} - \rho(\widetilde{b})\bigg)\| &=& \|\sum_{i=1}^{m} \beta_i \sigma(\widetilde{a_i}) - \sigma(\rho(\widetilde{b}))\| \\ &=& \| \sum_{i=1}^{m} \beta_i \sigma(\widetilde{a_i}) - \sigma(\widetilde{b}) \|\\ & =& \| \sum_{i=1}^{m} \beta_i \sigma(\widetilde{a_i}) - \widetilde{w} \|.
		\end{eqnarray*}
		Therefore,
		\[
		\| \sum_{i=1}^{m} \beta_i\sigma(\widetilde{a_i}) - \sum_{i=1}^{m} \alpha_i\sigma(\widetilde{a_i}) \| \leq \| \sum_{i=1}^{m} \beta_i\sigma(\widetilde{a_i}) - \widetilde{w} \|.
		\]
		In other words, $\sum_{i=1}^{m} \alpha_i\sigma(\widetilde{a_i})$ is a best coapproximation to $\widetilde{w}$ out of $ \sigma(\mathbb{Y}).$ This establishes the necessary part of the theorem. 
		
		\smallskip
		
		To prove the sufficient part, let $ \widetilde{\eta} \in \ell_1^n.$ Since $\sigma(\mathbb{Y}) $ is a coproximinal subspace, there exist  $ \alpha_1, \alpha_2, \ldots, \alpha_m \in \mathbb{R}$ such that  $ \sum_{i=1}^{m} \alpha_i \sigma(\widetilde{a_i}) $ is a best coapproximation to $ \sigma(\widetilde{\eta})$ out of $ \sigma(\mathbb{Y}).$ Therefore, for any $ \beta_1, \beta_2, \ldots, \beta_m \in \mathbb{R},$
		\begin{eqnarray*}
			\|\sum_{i=1}^{m} \beta_i \widetilde{a_i} - \sum_{k=1}^{m} \alpha_i \widetilde{a_i} \|  & = & \| \sum_{i=1}^{m} \beta_i \sigma(\widetilde{a_i})-  \sum_{i=1}^{m} \alpha_i \sigma(\widetilde{a_i}) \| \\ & \leq & \| \sum_{i=1}^{m} \beta_i \sigma(\widetilde{a_i}) - \sigma(\widetilde{\eta}) \|  \\ & \leq &  \| \sum_{i=1}^{m} \beta_i \widetilde{a_i} - \widetilde{\eta} \|. 
		\end{eqnarray*}
		Therefore, $\sum_{i=1}^{m} \alpha_i \widetilde{a_i}$ is a best coapproximation to $\widetilde{\eta}$ out of $ \mathbb{{Y}}.$ This completes the theorem.
	\end{proof}
	
	Our final result in this section reads as follows.
	
	\begin{theorem}\label{co-Chebyshev}
		Let  $\mathbb{Y}$ be a subspace of $\ell_1^n$ and let $ \mathcal{A}= \{ \widetilde{a_1}, \widetilde{a_2}, \ldots, \widetilde{a_m}\}$ be a basis of $\mathbb{Y}$ with $|\mathcal{Z}_{\mathcal{A}}| = r > 0.$ Then  $ \mathbb{{Y}} $ is  not a  co-Chebyshev subspace of $ \ell_1^n.$ 
	\end{theorem}
	
	\begin{proof}
		Let $\widetilde{b}=(b_1, b_2, \ldots, b_n) \not \in \mathbb{Y}$ be such that $ \rho(\widetilde{b})= \theta \in \ell_1^n,$ which implies that $b_i=0,$ for any $i \notin \mathcal{Z}_{\mathcal{A}}.$  Now for any $ \alpha_1, \alpha_2, \ldots, \alpha_m \in \mathbb{R}$ satisfying $ \| \sum_{i=1}^{m} \alpha_i \widetilde{a_i} \| \leq \|\widetilde{b}\|$ and  for any $ \beta_1, \beta_2, \ldots, \beta_m \in \mathbb{R},$ it is easy to observe that 
		\begin{eqnarray*}
			\|\sum_{i=1}^{m} \beta_i \widetilde{a_i} - \sum_{i=1}^{m} \alpha_i \widetilde{a_i} \|  & \leq & 	\|\sum_{i=1}^{m} \beta_i \widetilde{a_i} \| + \| \sum_{i=1}^{m} \alpha_i \widetilde{a_i} \| \\ & \leq & 	\|\sum_{i=1}^{m} \beta_i \widetilde{a_i}\| + \|\widetilde{b}\|.
		\end{eqnarray*}
		We also note that
		\begin{eqnarray*}
			\|\sum_{i=1}^{m} \beta_i \widetilde{a_i} - \widetilde{b} \| &=& \sum_{j=1}^{n} \bigg|  \bigg( \sum_{i=1}^{m} \beta_ia_j^i - b_j   \bigg)\bigg|\\ &=& 
				\sum_{j \notin \mathcal{Z}_{\mathcal{A}}} \bigg|  \bigg( \sum_{i=1}^{m} \beta_ia_j^i - b_j   \bigg)\bigg| + \sum_{j \in \mathcal{Z}_{\mathcal{A}}} \bigg|  \bigg( \sum_{i=1}^{m} \beta_ia_j^i - b_j   \bigg) \bigg| \\
				&=& \sum_{j \notin \mathcal{Z}_{\mathcal{A}}} \bigg|  \sum_{i=1}^{m} \beta_ia_j^i \bigg| +  \sum_{j \in \mathcal{Z}_{\mathcal{A}}} | b_j|\\
			&	=& \|\sum_{i=1}^{m} \beta_i \widetilde{a_i} \|+\| \widetilde{b} \|.
		\end{eqnarray*}
		Therefore, 
		\[
		\|\sum_{i=1}^{m} \beta_i \widetilde{a_i} - \sum_{i=1}^{m} \alpha_i \widetilde{a_i}\| \leq  \|\sum_{i=1}^{m} \beta_i \widetilde{a_i} - \widetilde{b} \|.
		\]
		In other words, for any $ \alpha_1, \alpha_2, \ldots, \alpha_m \in \mathbb{R} $ such that $ \| \sum_{i=1}^{m} \alpha_i \widetilde{a_i} \| \leq \|\widetilde{b}\|,  ~  \sum_{i=1}^{m} \alpha_i \widetilde{a_i} $ is a best coapproximations to $ \widetilde{b} $ out of $ \mathbb{Y}.$ Therefore, $\mathbb{{Y}}$ is not a co-Chebyshev subspace of $\ell_1^n.$ This establishes the theorem. 
		
	\end{proof}

We end this article with  examples of   both   coproximinal and not coproximinal subspaces for which the zero set is non-empty.
	
	\begin{example}\label{newexample}
		It can be easily verified by using the methods developed in this article that $ \mathbb{Y}_1$ is a coproximinal subspace of $\ell_1^7,$ whereas $ \mathbb{Y}_2$ is not, where 
		\[
		\mathbb{Y}_1= span\{ (1, 1, 2, 0, 4, -2, 0), (1, 2, 2, 0, 4, -4, 0)\},\]
		\[	\mathbb{Y}_2= span\{ (1,0, 2, 3, -1, -2, 0), (-1, 0, 1, 0, 1, -1, 0)\}.
		\] 
		Moreover, 
		\[
		\mathbb{{Y}}_3= span \{ (1, 1, 2, 4, -2), (1, 2, 2, 4, -4)\}
		\]
		is a co-Chebyshev subspace of $\ell_1^5,$ but $\mathbb{{Y}}_1$ is not co-Chebyshev.
	\end{example}
	
	  \section*{Conclusions}

	 The computational difficulty in solving the best coapproximation problem arises essentially from the non-linear nature of the inequalities associated with it. We have illustrated in this article that by applying Birkhoff-James orthogonality techniques, it is possible to reduce the much harder non-linear problem into a  system  of linear equations (see Theorem \ref{characterization}).\\

	 We have presented explicit examples to highlight the different possibilities for subspaces of $\ell_1^n,$ from the perspective of best coapproximation. Indeed, it follows from our observations that the newly introduced ``zero set'' of a subspace plays a fundamental role in the whole scheme of things (see Example \ref{example}, Corollary \ref{cop}). We have also explored the relationship between coproximinal subspaces and co-Chebyshev subspaces of $\ell_1^n,$ depending on the zero sets of the concerned subspaces. In particular, it is to be noted that there exists a coproximinal subspace of $\ell_1^n,$ which is not co-Chebyshev (see Example \ref{newexample}).\\

	In view of the methods developed here, applications of the concept of orthogonality in solving the best coapproximation problem in Banach spaces seem to be a promising direction of research, resulting in efficient algorithms which are computationally advantageous. We have presented several numerical examples in support of this, in the specific setting of $\ell_1^n$ spaces. For an analogous approach to the best coapproximation problem in $\ell_{\infty}^n$ spaces, the readers are referred to the recent article \cite{SSGP}. As a matter of fact, it may be interesting to apply Birkhoff-James orthogonality towards obtaining computationally efficient algorithmic solutions to the said problem, in the setting of other classical Banach spaces, such as the $ \ell_p^n $ spaces, where $ 1 < p(\neq 2)  < \infty. $ \\

	\noindent \textbf{Declarations.} \\
	The authors have no competing interests to declare that are relevant to the content of this article.


\begin{thebibliography}{99}
		
		\bibitem{B} Birkhoff, G.,  \textit{Orthogonality in linear metric spaces}, Duke Math. J., \textbf{1} (1935), 169--172.
		
		\bibitem{FF} Franchetti, C., Furi, M., \textit{Some characteristic properties of real Hilbert spaces},  Rev. Roumaine Math. Pures Appl., \textbf{17} (1972), 1045-1048.
		
		\bibitem{J}James, R. C., \textit{Orthogonality and linear functionals in normed linear spaces}, Trans.  Amer.  Math. Soc., \textbf{61} (1947), 265-292.
		
		\bibitem{LT} Lewicki, G., Trombetta, G., \textit{Optimal and one-complemented subspaces}, Monatsh. Math., \textbf{153} (2008), 115-132.
		
		
		\bibitem{N} Narang, T. D., \textit{On best coapproximation in normed linear spaces}, Rocky Mountain J. Math., \textbf{22} (1992), 265-287.
		
		\bibitem{PS} Papini, P. L., Singer, I., \textit{Best coapproximation in normed linear spaces}, Mh. Math., \textbf{88}  (1979), 27-44.
		
		\bibitem{RS} Rao, G.S., Swaminathan, M., \textit{Best coapproximation and Schauder bases in Banach spaces}, Acta Sci. Math., \textbf{54} (1990), 339-354.
		
		\bibitem{S} Sain, D., \textit{On best approximations to compact operators}, Proc. Amer. Math. Soc., \textbf{149} (2021), 4273-4286
		
		
		\bibitem{SSGP} Sain, D., Sohel, S., Ghosh, S., Paul, K., \textit{On
			best coapproximations in subspaces of diagonal matrices},  Linear Multilinear Algebra, DOI:10.1080/03081087.2021.2017835
		
		
	
	\end{thebibliography}
\end{document}